\documentclass[reqno,12pt]{amsart}

\usepackage[T1]{fontenc}
\usepackage{amssymb,amsfonts,amsbsy,graphicx,setspace,color}
\usepackage[usenames,dvipsnames]{pstricks}
\makeatletter

\everymath{\displaystyle} %pour tout afficher en displaystyle !

 \theoremstyle{plain}
 \newtheorem{thm}{Theorem}[section]  
 \newtheorem{lmm}[thm]{Lemma}  
 \newtheorem{corol}[thm]{Corollary}
 \numberwithin{equation}{section} %% Comment out for sequentially-numbered
 \numberwithin{figure}{section} %% Comment out for sequentially-numbered
 \newtheorem{prop}[thm]{Proposition}
 \theoremstyle{remark}
 \newtheorem{rmk}[thm]{Remark}

 \newtheorem*{acknowledgement*}{Acknowledgement}
 \theoremstyle{definition}
 \newtheorem{defi}[thm]{Definition}

\def\rc{\mathcal R}

\def\Pb{\mathcal P}

\newcommand{\Co}{{\bf\mathfrak{C}}}

\newcommand{\R}{\mathbb R}
\newcommand{\Compl}{\mathbb C}

\newcommand{\N}{\mathbb N}
\newcommand{\Z}{\mathbb Z}

\newcommand{\A}{\mathcal A}
\newcommand{\bQ}{{\bf Q}}
\newcommand{\Perm}{{\bf {\mathcal S}}}
\def\rn{r_n}
\def\kn{k_n}
\def\bx{{\bf x}}
\def\by{{\bf y}}


\date{15 March 2015}
\author{Luigi De Pascale}
\title[Duality for the Coulomb cost]{Optimal Transport with Coulomb cost. Approximation and duality.}

\begin{document}
\maketitle

\begin{abstract}  We revisit the duality theorem for multimarginal  optimal transportation problems. In particular, we focus on the Coulomb cost. We use a discrete approximation to prove 
equality of the extremal values and some careful estimates of the approximating sequence to prove existence of maximizers for the dual problem (Kantorovich's potentials). Finally we observe that the same strategy can be applied to a more general class of costs and that a classical results on the topic cannot be applied here.
 \end{abstract}

\section{Introduction}
This paper deals with the following variational problem.
Let $\rho \in \Pb(\R^3)$  be a probability measure ($\rho$ will be called electronic density) and let 
$$ c(x_1, \dots, x_N)=\sum_{1\leq i <j \leq N} \frac{1}{|x_i-x_j|}$$ 
be the Coulomb cost. Consider the set $\Pi(\rho)=\{ P \in \Pb(\R^{3N})\ | \ \pi^i_\sharp P=\rho\},$ where $\pi^i$ denotes the projection on the $i-$th copy of $\R^3$ and $\pi^i_\sharp P$ is the push-forward measure.
We aim to minimize
\begin{equation}\label{main}
\min_{\Pi(\rho)} \int_{\R^{Nd}} c(x_1,\dots,x_N) d P(x_1, \dots, x_N).
\end{equation}
This problem fits in the general framework of multimarginal optimal transportation problems. In particular it is the multimarginal optimal transportation problem with all the $N$ marginals coinciding with $\rho$ and with Coulomb cost.

In contrast with the classical two-marginal optimal transportation problems, the theory of multimarginal optimal transportation problems 
is still at the beginning and many relevant open problems need to be studied.
Some general results are available in \cite{carlier2003class, kellerer1984, pass2012local,pass2011uniqueness,racrus1998}, 
results for special costs are available, 
for example in \cite{gangbo1998optimal} for the quadratic cost with some generalization in \cite{heinich2002probleme}, 
and in \cite{carlier2008optimal} for the determinant. Some new applications are emerging for example in \cite{ghoussoub2011self}.

In the particular case of the Coulomb cost there are also many other questions related to the applications. 
Recent results on the topic are contained in  \cite{buttazzo2012,colombo2013equality,colombo2013multimarginal,cotar2013density,friesecke2013n} and some of them will be described better in subsequent sections.

The literature quoted so far is not at all exaustive and we refer the reader to the bibliographies of the cited papers for a more detailed picture. However, in the author's opinion, at the moment the wealth of problems obscures the body of known  results.
In this paper we will focus on the Kantorovich duality for problem (\ref{main}). We denote by $\rho^N$ the product measure 
$\underbrace{\rho\otimes\dots\otimes\rho}_{N-times}$.  We will prove that
\begin{multline}\label{firstdual}
\min_{\Pi(\rho)} \int_{\R^{Nd}} c(x_1,\dots,x_N)  d P(x_1, \dots, x_N)=
\sup\{N \int u d\rho \ : \ u \in L^1_\rho, \\ 
u(x_1)+\dots+u(x_N) \leq c(x_1,\dots,x_N), \ \ \rho^N -a.e. \}
\end{multline}
and that the right-hand side of (\ref{firstdual}) admits a maximizer which is, in particular, bounded.
Thanks to the symmetries of the problem we also have that the right-hand side of (\ref{firstdual}) coincide with 
\begin{equation*}
\sup\{\sum_{i=1}^N \int u_i d\rho \ : \ u_i \in L^1_\rho,\  u_1(x_1)+\dots+u_N(x_N) \leq c(x_1,\dots,x_N), \ \ \rho^N -a.e. \}.
\end{equation*}
Infact, this last $\sup$ is a priori bigger then (\ref{firstdual}).  Since for any admissible $N$-tuple $(u_1, \dots,u_N)$ the function $\displaystyle  u(x)=\frac{1}{N} \sum_{i=1}^N u_i(x)$ is admissible for the previous problem, equality holds.

Some of the basic ideas originated in the paper \cite{kellerer1984}. However, in Remark \ref{nokellerer}, we will show that Theorem 2.21 of \cite{kellerer1984}
does not apply to the Coulomb cost in order to prove that a maximizer of the dual problem exists.
The tools will be $\Gamma$-convergence and some careful estimates of the maximizers of the approximating problems.

We remark that a necessary but not sufficient {\bf assumption} for (\ref{main}) to be finite is that 
$\rho$ is not concentrated on a set of cardinality $\le N-1$. We will assume this whenever needed.

We will adopt the notations $\bx=(x_1,\dots,x_N) \in \R^{3N}$ so that $x_i \in \R^3$ for each $i \in \{1,\dots,N\}.$ 
And $\rho-\inf$ will denote both the essential $\rho$ infimum or the essential $\rho^{N-1}$ infimum depending on the number of variables involved. 

\subsection{Motivations}
The main object in the quantum mechanical modeling of a particle with $N$ electrons is a wave-function $\psi$, i.e. an element of 
$$\A:=\{\psi \in H^1((\R^3\times \Z_2)^N, \Compl)\ : \ \|\psi\|_{L^2}=1\}.$$
The space $(\R^3\times \Z_2)^N$ is the configuration space for the $N$ electrons. In fact, the state of each electron is individuated
by the position in $\R^3$ and the spin in $\Z_2$. While it is easy to understand the interpretation of the position variable, the notion of spin 
is slightly more difficult to grasp. We may roughly explain it as follows: when a magnetic field is applied to the electron of position  $x_0$ velocity $v_0$, the electron may be deflected in two different ways which are characterized by the two possible spins. The name is then reminiscent of the behavior of a spinning top which hits a wall with a certain  translational velocity in a point $p_0$ and may rebound in two different directions depending on the whirling direction. However this analogy, although suggestive, is not the historical one nor can be considered physically meaningful. 

It is easy, once we free ourself from this last analogy, to imagine the possibility of a spin variable with values in $\Z_k$ or other spaces.\\

%\cm{[This explanation of the "origin of the name" spin is quite different from the usual one coming from mathematical physics. The modern/``physically ok'' explanation for the spin is via %representation theory, related to Dirac's work and is purely quantum-mechanical. Also the angular momentum cited by the author has a Lie group interpretation, but I'm not sure how it i%s relevant. On the other hand, historically the (a posteriori physically misleading!) name ``spin'' must have come from Stern and Gerlach's mental picture: they modeled electrons as %circuits where charge was ``spinning'', generating a magnetic moment, which interacted with a constant external magnetic field deflecting the atoms in a way dependent on the %spinning and, surprisingly at the time, quantized. I suggest that the author should either change this part or make it clear that the etymology he proposes, although perhaps simpler, is %not the historical one.]}

The quantity 
$$ |\psi((x_1,\alpha_1), \dots, (x_N, \alpha_N))|^2$$
represents the probability that the $N$ electrons occupy the state $((x_1,\alpha_1), \dots, \\ (x_N, \alpha_N))$ and then, since the electrons are indistinguishable $\psi$ satisfies
$$ |\psi((x_{\sigma(1)},\alpha_{\sigma(1)}), \dots, (x_{\sigma(N)}, \alpha_{\sigma(N)}))|^2=|\psi((x_1,\alpha_1), \dots, (x_N, \alpha_N))|^2$$
for all permutations $\sigma$ of the $N$ electrons.
This brings to introduce the distinction between Fermionic and Bosonic particles, however such distinction will not have relevance here since we will discuss 
duality for a relaxed problem. The distinction may be instead relevant when discussing the relaxation process.
A particle is {\bf Fermionic} if 
$$ \psi((x_{\sigma(1)},\alpha_{\sigma(1)}), \dots, (x_{\sigma(N)}, \alpha_{\sigma(N)}))=sgn(\sigma) \psi((x_1,\alpha_1), \dots, (x_N, \alpha_N)),$$
and {\bf Bosonic} if
$$ \psi((x_{\sigma(1)},\alpha_{\sigma(1)}), \dots, (x_{\sigma(N)}, \alpha_{\sigma(N)}))=\psi((x_1,\alpha_1), \dots, (x_N, \alpha_N)).$$
Since the spin will not play any role in the rest of the paper we will 
consider wave-functions depending only on the positions of the electrons. Including the dependence on the spin is just a matter of adding a summation over the two possible values in the formulas  (below we give an example in the case of kinetic energy).
In the simplest situation the electrons move with a certain velocity
 while interacting with the nuclei (or, equivalently, with an external potential) and with each other. The interaction with the nuclei and the interaction between electrons are of Coulombian nature. So if we assume that there are $M$ nuclei with charge $Z_i$ and position $R_i$ the interaction potential in the position $x$ will be $v(x)= - \sum_{i=1} ^M \frac{Z_i}{|x-R_i|}$.
Then the energy of a particular state of the particle is made of three parts:\\
the {\bf Kinetic energy}\footnote{if we want to include the "spin coordinate" in the state of the particle then 
$$T[\psi]= \sum_{\alpha_1,\dots,\alpha_n=0}^1 \frac{1}{2} \int_{\R^{3N}} |\nabla \psi((x_1,\alpha_1), \dots, (x_N, \alpha_N))|^2 dx_1 \dots d x_N.$$}
$$T[\psi]=\frac{1}{2} \int_{\R^{3N}} |\nabla \psi(x_1), \dots, x_N)|^2 dx_1 \dots d x_N, $$
the {\bf electron-nuclei} interaction energy 
$$ V_{ne}[\psi]= \sum_{i=1}^N \int_{\R^{3N}} v(x_i)|\psi(x_1,\dots x_N)|^2 dx_1 \dots dx_N,$$
the {\bf electron-electron} interaction energy
$$ V_{ee}[\psi]= \sum_{1 \leq i <j \leq N} \int_{\R^{3N}} \frac{1}{|x_i-x_j|} |\psi(x_1,\dots,x_N)|^2 dx_1 \dots dx_N.$$

A relevant quantity in quantum mechanics, quantum chemistry and other disciplines is the ground state energy of an atom
$$E_0= \min_{\psi \in \A} T[\psi]+V_{ee}[\psi]+V_{ne} [\psi].$$
\begin{rmk} 
The special form of the electron-nuclei interaction is widely irrelevant in the expression of $V_{ee}$ above. So one could just include a term of the same form with an external potential $v$.  
 \end{rmk}
 In order to compute (numerically) $E_0$ one needs to solve the Schr\"odinger equation in $3N$ dimensions 
and this is very costly even for a small number of electrons. 
A less costly alternative is represented by the {\bf Density-Functional Theory}
introduced first by Hohenberg and Kohn in \cite{hohenberg1964} and then Kohn and Sham in \cite{kohn1965}. 
At the beginning of the theory the mathematical foundations of DFT were very weak. 
The paper which started to put DFT on solid mathematical foundations is, to my knowledge, \cite{lieb1983}.

The main object in DFT is the electronic density $\rho$ which is obtained by integrating out $N-1$ particles
$$\rho(x)=\int_{\R^{3(N-1)}} |\psi(x,x_2,\dots,x_N)|^2 dx_2\dots dx_N$$ 
and it represents the probability distribution of a single electron. 
The relation between $\psi$ and $\rho$ will be denoted by $\psi \downarrow \rho$. 
In particular $\rho$ is always a probability density on $\R^3$ independently of 
the number of electrons. The basic idea of DFT is to express everything in term 
of the electronic density $\rho$ instead of the wave function $\psi$.
It is easy to express the  {\bf electron-nuclei} interaction energy in terms of $\rho$. We have
$$ V_{ne}[\psi]= N \int_{\R^3} v(x) d \rho(x).$$
Then, following Hohenberg and Kohn, we write
\begin{equation}\label{HK}
E_0= \min_\rho \{ F_{HK}(\rho)+  N \int_{\R^3} v(x) d \rho(x) \}
\end{equation}
where 
$$  F_{HK}(\rho)= \min_{\psi \downarrow \rho}\{ T[\psi]+V_{ee}[\psi] \},$$ 
is called the universal Hohenberg-Kohn functional ({\it universal} since it does not depends on the specific particle).
At the beginning of the DFT several mathematical questions needed to be settled in order to have a meaningful 
mathematical theory.
For example: Which  is the correct minimization domain? 
Does the functional $F_{HK}$ enjoy some properties which make the variational problem amenable? And so on.
As we said, Lieb in \cite{lieb1983} started to look at these questions. In particular, Lieb described the set of admissible $\rho$
$$ H=\{\rho \ | \ 0 \leq \rho, \ \int \rho =1, \ \sqrt{\rho} \in H^1 (\R^3)  \}.$$
Writing explicitly $F_{HK}$ is not possible, then Kohn and Sham considered 
$$  F_{KS}(\rho)= \min_{\psi \downarrow \rho}\{ T[\psi]\}.$$
In the Bosonic case it is possible to express  
$$  F_{KS}(\rho)= N\int |\nabla \sqrt{\rho}|^2 dx.$$
In the fermionic case 
$$ \int |\nabla \sqrt{\rho}|^2 dx \leq \frac{1}{N}F_{KS}(\rho) \leq (4 \pi)^2 N^2\int |\nabla \sqrt{\rho}|^2 dx$$
(But the optimal  constant in the inequality above, to my  knowledge is still not known).
Then one may rewrite
$$ F_{HK}=F_{KS}+F_{xc} $$
where $F_{xc} $ is called {\it exchange-correlation energy} it needs to be determined or at 
least approximated and it is the term which keeps into account the interaction between electrons.

Let us denote by $\Co(\rho)$ the minimal value in (\ref{main}). Since for every $\psi \downarrow \rho$ the measure $|\psi|^2 dx_1\dots dx_N \in \Pi(\rho)$ we have 
$$\Co(\rho)\leq  V_{ee}[\psi].$$
The functional $\Co(\rho)$ appear in DFT in several ways. The first and perhaps the most elementary appearance is in the estimate from below 
$$F_{KS}(\rho)+\Co(\rho) \leq  F_{HK}(\rho).$$
This estimate is useful in the variational study of DFT and is is also the basis for the so called KS-SCE DFT (Kohn and Sham, Strictly Correlated Electrons Density-Functional Theory) in which $\Co(\rho)$ is interpreted as an approximation of $F_{xc} $ for particles in which the $electron-electron$ interaction is relevant. This is a very rapidly 
developing domain started in \cite{seidl1999strong} (see also \cite{gori2009density,seidl1999strictly,seidl2007strictly} but we are far from being exhaustive).

Again $\Co(\rho)$ appear when considering the SCE-DFT (Strictly Correlated Electrons-Density-Functional Theory), 
which is the analogous of DFT in a regime in which the $electron-electron$ interaction is preponderant. 
This means writing  
$$ F_{HK}=\Co+F_{kd} ,$$
where $F_{kd}$ is called {\it kinetic-decorrelation energy}, it needs to be determined and it is the term which keeps into account the kinetic energy of the particle (see, for example, \cite{gori2009density,gori2010density}). 

The third appearance is given in \cite{cotar2013density} where the authors proved that $\Co$ is the semiclassical limit of the DFT in the case of a 
two electrons particle. It is not clear if this last result also holds for more than two electrons.

Finally we quote a direct application of the duality theorem we are considering. In the numerical approximations of all the problems above one need to compute $\Co(\rho)$. Before the link with Optimal Transport problems was discovered this approximation was done by computing the co-motion functions which are the analogue of the optimal transport maps. However, more recently, Mendl and Lin in \cite{mendl2013kantorovich}  introduced a numerical method which use the Kantorovich potential to compute the value.

\section{Preliminary results and tools}

\subsection{Definition of $\Gamma$-convergence and basic results}

A crucial tool that we will use throughout this paper is $\Gamma$-convergence. 
All the details can be found, for instance, in Braides's book \cite{braides2002gamma} or in the classical book by 
Dal Maso \cite{dal1993introduction}. In what follows, $(X,d)$ is a metric space or a topological space equipped with a convergence.

\begin{defi} Let $(F_n)_n$ be a sequence of functions $X \mapsto \bar\R$. We say that $(F_n)_n$ $\Gamma$-converges to 
$F$ and we write $F_n \xrightarrow[n]{\Gamma} F$ if for any $x \in X$ we have
\begin{itemize}
\item for any sequence $(x_n)_n$ of $X$ converging to $x$
$$ \liminf\limits_n F_n(x_n) \geq F(x) \qquad \text{($\Gamma$-liminf inequality);}$$
\item there exists a sequence $(x_n)_n$ converging to $x$ and such that
$$ \limsup\limits_n F_n(x_n) \leq F(x) \qquad \text{($\Gamma$-limsup inequality).} $$
\end{itemize} \end{defi}

This definition is actually equivalent to the following equalities for any $x \in X$:
$$ F(x) = \inf\left\{ \liminf\limits_n F_n(x_n) : x_n \to x \right\} = \inf\left\{ \limsup\limits_n F_n(x_n) : x_n \to x \right\} $$
The function $x \mapsto  \inf\left\{ \liminf\limits_n F_n(x_n) : x_n \to x \right\}$ is called $\Gamma$-liminf of the sequence $(F_n)_n$ and the other one its $\Gamma$-limsup. A useful result is the following (which for instance implies that a constant sequence of functions does not $\Gamma$-converge to itself in general).

\begin{prop}\label{lsc} The $\Gamma$-liminf and the $\Gamma$-limsup of a sequence of functions $(F_n)_n$ are both lower semi-continuous on $X$. \end{prop}

The main interest of $\Gamma$-convergence resides in its consequences in terms of convergence of minima:

\begin{thm} \label{convminima} Let $(F_n)_n$ be a sequence of functions $X \to \bar\R$ and assume that $F_n \xrightarrow[n]{\Gamma} F$. Assume moreover that there exists a compact and non-empty subset $K$ of $X$ such that
$$ \forall n\in N, \; \inf_X F_n = \inf_K F_n $$
(we say that $(F_n)_n$ is equi-mildly coercive on $X$). Then $F$ admits a minimum on $X$ and the sequence $(\inf_X F_n)_n$ converges to $\min F$. Moreover, if $(x_n)_n$ is a sequence of $X$ such that
$$ \lim_n F_n(x_n) = \lim_n (\inf_X F_n)  $$
and if $(x_{\phi(n)})_n$ is a subsequence of $(x_n)_n$ having a limit $x$, then $ F(x) = \inf_X F $. 
\end{thm}

Since we will deal also with dual problems we need the analogue of $\Gamma$-convergence for maxima, which is sometimes called $\Gamma^+$-convergence.

\begin{defi} Let $(F_n)_n$ be a sequence of functions $X \mapsto \bar\R$. We say that $(F_n)_n$ $\Gamma^+$-converges to $F$, 
and we write $F_n \xrightarrow[n]{\Gamma^+} F$ if for any $x \in X$ we have
\begin{itemize}
\item for any sequence $(x_n)_n$ of $X$ converging to $x$
$$ \limsup\limits_n F_n(x_n) \leq F(x) \qquad \text{($\Gamma^+$-limsup inequality);}$$
\item there exists a sequence $(x_n)_n$ converging to $x$ and such that
$$ \liminf\limits_n F_n(x_n) \geq F(x) \qquad \text{($\Gamma^+$-liminf inequality).} $$
\end{itemize} \end{defi}

The natural generalizations of Prop. \ref{lsc} and Th. \ref{convminima} hold with upper semicontinuity, maximum values and maximum points replacing lower semicontinuity, minimum values and minimum points.

\subsection{Permutation invariant costs}
The Coulomb cost $c$ as well as the approximations that we will consider are permutation invariant in the sense that
$$ c(x_{\sigma(1)}, \dots,x_{\sigma(N)})=c(x_1,\dots, x_N), \ \ \ \forall \sigma \in \Perm^N.$$

Several simplifications are permitted by this invariance.

\section{Duality}
Denote by
\begin{equation*}
\chi_K(x)=\left\{\begin{array}{ll}
1 & \mbox{if} \ x \in K,\\
0 & \mbox{otherwise}.
\end{array}\right.
\end{equation*}
Introduce the set of {\it elementary} functions 
$${\mathcal E} (\R^{3N})=\{ \varphi:\R^{3N}\to \R \ \mbox{measurable}\ : \ \varphi(x)= \sum_{i=1}^k a_i \chi_{A^1_i\times \dots \times A^N_i }\}$$
for suitables constants $a_i$ and Borel sets $A^i_k \subset \R^3$.
We consider the following approximation of $c$.

\begin{prop} There exists a sequence of costs $c_n$ such that
\begin{enumerate}
\item $0 \leq c_n \leq c$,
\item $c_n \in {\mathcal E}(\R^{3N})$,
\item $c_n \nearrow c$,
\item $c_n$ is  l.s.c.,
\item $c_n$ is permutation invariant.
\end{enumerate}
\end{prop}
\begin{proof}
Let $(a_1, \dots, a_N) \in \Z^{3N}$  and define $a-1:= (k-1,p-1,q-1)$ for $a=(k,p,q) \in \Z^3$ and (with a slight abuse of notations) $(x,y):=(x_1,y_1)\times(x_2,y_2)\times(x_3,y_3)$ for
$x,y\in \R^3$. Then $\forall n \in \N$ define the 
interior of the diadyc cube by
$$Q^n(a_1,\dots, a_N)=(\frac{a_1 -1}{2^n},\frac{a_1}{2^n})\times \dots \times(\frac{a_N -1}{2^n},\frac{a_N}{2^n}).$$
and for all positives $R$
$$ \bQ(R)= \underbrace{[-R,R]\times \dots \times [-R,R]}_{3N-\mbox{times}}.$$
Define $c_n$ as follows
\begin{equation*}
c_n(\bx)=\left\{\begin{array}{ll}
\inf_{Q^n(a_1,\dots, a_N)} c(\bx) & \mbox{if} \ \bx \in Q^n(a_1,\dots, a_N) \ \mbox{and}\ Q^n(a_1,\dots, a_N)\subset \bQ(n),\\
0 & \mbox{if} \ \bx \not \in \bQ(n).
\end{array}
\right. 
\end{equation*}
Then extend $c_n$ to the entire space by relaxation i.e.
\begin{multline*}
 c_n(\bx)=\inf \{ \liminf_{k \to \infty} c(\bx^k) \ : \ \lim_{k \to \infty} \bx^k =\bx, \ \mbox{and}\\\forall k, \ \bx^k \in Q^n(a_1,\dots, a_N) \ 
\mbox{for some}\ (a_1, \dots, a_N) \in \Z^{3N}\}.
\end{multline*}
With this definition properties $(1), (3)$ and $(4)$ above are straightforward.
To prove $(2)$ it is enough to remark that $c_n$ is constant on $i$ dimensional, relatively open faces of $Q^n(a_1,\dots, a_N)$
and that such faces are products of intervals (which may degenerate to a point). 
Finally $(5)$ holds since 
$$ \inf_{Q^n(a_1,\dots, a_N)} c=  \inf_{Q^n(a_{\sigma(1)},\dots, a_{\sigma(N)})}c $$
for every $\sigma \in \Perm^N$.
\end{proof}

\begin{rmk} Without loss of generality we can write 
$$c_n(\bx) = \sum_{i=1}^{k_n} a_i \chi_{A^1_i}(x_1) \dots  \chi_{A^N_i}(x_N) $$
with sets $A^k_i$ such that if  $ A^k_i \cap A^k_j\neq \emptyset$ then  $ A^k_i =A^k_j,$ for all $k\in \{1,\dots,N\}.$ 
\end{rmk}

Define $K(c_n, \cdot): \Pb(\R^{3N}) \to [0, +\infty]$ as
$$K(c_n, P)= \left\lbrace \begin{array}{ll}
\int c_n dP & \mbox{if} \ P \in \Pi(\rho), \\
+\infty &\mbox{otherwise,}
\end{array}\right.
$$
and $D(c_n, \cdot): L^1_\rho  \to \overline{\R}$ as
$$D(c_n, v)= \left\lbrace \begin{array}{ll}
N\int v d\rho & \mbox{if} \  v(x_1)+\dots +v(x_N) \leq c_n(x_1, \dots, x_N),\ \rho^N-a.e. \\
-\infty &\mbox{otherwise.}
\end{array}\right.
$$
Also in this case, maximizing $D(c_n, v)$ is equivalent to maximize 
$$\left\lbrace \begin{array}{ll}
\sum_{i=1}^N\int v_i d\rho & \mbox{if} \  v_1(x_1)+\dots +v_N(x_N) \leq c_n(x_1, \dots, x_N),\ \rho^N-a.e. \\
-\infty &\mbox{otherwise.}
\end{array}\right.
$$

\begin{prop}\label{gammamisure} The functionals $K(c_n, \cdot)$ are equicoercive and 
$$ K(c_n, \cdot)\stackrel{\Gamma}{\rightarrow} K(c, \cdot),$$
with respect to the $w^*$-convergence of measures.
\end{prop}
\begin{proof}
Equicoercivity follows from the fact that $\Pi (\rho)$  is $w^*$-compact.
Since $K(c_n,\dot)$ is non-decreasing (in $n$) and $K(c, \cdot)$ is l.s.c.,
$\Gamma$-convergence is a standard fact.  We report the proof for the sake of completeness.
Let $P_n \stackrel{*}{\rightharpoonup} P$ and fix $m \in \N$.
For $m <n$ 
$$K(c_m,P_n) \leq K(c_n, P_n) .$$
Since $c_m$ is lower semi-continuous 
$$ K(c_m, P ) \leq \liminf_{n \to \infty} K(c_m, P_n) \leq  \liminf_{n \to \infty} K(c_n, P_n) .$$
And since $K(c_m,P) \to K(c,P)$ we obtain 
$$ K(c, P ) \leq \liminf_{n \to \infty} K(c_n, P_n).$$
For what concerns the $\Gamma-\limsup$ inequality, since $c_n \nearrow c$ it is enough to consider $P_n=P$ for all $n \in \N$ and to apply 
the Monotone Convergence Theorem.
\end{proof}

\begin{lmm}\label{ubounded} Assume that $\rho$ is not concentrated on a set of cardinality $\le N-1$. If $u \in L^1_\rho$ and
$$u(x_1)+\dots+u(x_N) \leq c(x_1,\dots,x_N), \ \ \ \rho^N-a.e.,$$
then there exists $k\in \R$ such that $u \leq k$ $\rho$-a.e.
\end{lmm}
\begin{proof}
We have
$$u(x)\leq \rho- \inf_{x_2, \dots, x_N} \{c(x,x_2,\dots, x_N)-u(x_2)-\dots-u(x_N)\}, \ \ \rho-a.e. $$

Since $u \in L^1_\rho$ and $\rho$ has not $\le N-1$ atoms, we may consider $\underline{x}_1,\dots, \underline{x}_N$ with 
$\underline{x}_i\neq \underline{x}_j$ if $i \neq j$ such that $u(\underline{x}_i) \in \R$. We denote $l= \min\{u(\underline{x}_1), \dots, u(\underline{x}_N)\}$.
Consider $0<r$ such that $B(\underline{x}_i,r) \cap B(\underline{x}_j,r)=\emptyset$ if $i \neq j$.
Then any $x\in \R^3$ can belong to $B(\underline{x}_i,r)$ for at most one $i$. We may suppose without loss of generality that $i=1$.
It follows that
$$u(x) \leq c(x, \underline{x}_2,\dots, \underline{x}_N)- u(\underline{x}_2)- \dots- u(\underline{x}_N) \leq \frac{N(N-1)}{2} \frac{1}{r}-Nl:=k. $$
\end{proof}
\begin{rmk} In the previous proof $k$ clearly depends on $u$  through the constant $l$ and on $\rho$ through the constant $r$.
\end{rmk}

\begin{prop}\label{gammamassimi} The functionals 
$$ D(c_n, \cdot)\stackrel{\Gamma^+}{\rightarrow} D(c, \cdot),$$
with respect to the weak-$L^1$ convergence.
\end{prop}
\begin{proof}
We start with the $\Gamma^+-\liminf$ inequality.
Let $u\in L^1_\rho$ such that  $-\infty<D(c,u)$
then 
$$ u(x_1)+\dots+ u(x_N) \leq c (x_1, \dots, x_N), \ \ \rho^N-a.e. $$
and by Lemma \ref{ubounded} $u$ is bounded above. For an arbitrary  $0<\varepsilon$ there holds
$$ D(c,u) -\varepsilon < N \int_{\R^3} u d \rho. $$
Moreover by Lusin's theorem there exists $K \subset \R^3$ compact such that
$u_{|K}$ is continuous, $\rho(\R^3 \setminus K)<\varepsilon$ and 
$$ D(c,u)-\varepsilon < N\int_K u d \rho.$$
Since $c_n \nearrow c$ and are l.s.c we may apply  a Dini's type theorem and we obtain that there exists $n_0$ such that if $n_0 <n$ then
\begin{equation}\label{closeonk}
u(x_1)+ \dots+ u(x_N) -\varepsilon < c_n (x_1, \dots, x_N), \ \ \rho^N-a.e. 
\end{equation} 
on $K \times \dots \times K$.
We consider
$$ \overline{u}(x):= u(x)-\frac{\varepsilon}{N}-N (\sup u) \chi_{\R^3 \setminus K}(x).$$
Since $0 \leq c_n$ and (\ref{closeonk}) holds we have, for $n_0 <n$,
$$ \overline{u}(x_1)+ \dots + \overline{u}(x_N)\leq c_n(x_1, \dots,x_N), \ \ \rho^N-a.e. $$
Moreover if $M\in \N$ and $N^2 \sup u<M$ then
\begin{multline}
D(c,u)-(2+M)\varepsilon < N \int _{\R^3} u d \rho -\varepsilon -N^2\varepsilon \sup u \leq  N \int _{\R^3} u d \rho -\varepsilon -N^2\rho(\R^3 \setminus K) \sup u= \\
= N \int _{\R^3} \overline{u}d \rho.
\end{multline}
For what concernes the $\Gamma^+-\limsup$ inequality assume that 
$u_n \rightharpoonup u$ in $L^1_\rho$ and assume without loss of generality that
$$u_n(x_1)+ \dots+ u_n(x_N)  \leq c_n(x_1, \dots, x_N), \ \ \rho^N-a.e. $$
Then 
$$u_n(x_1)+ \dots+ u_n(x_N)  \leq c(x_1, \dots, x_N) $$
and 
$$ \lim_{n \to \infty} N\int u_n d \rho =N\int u d \rho.$$  
\end{proof}

In the next  subsection we will prove the needed compactness property.
\subsection{Estimates of the approximating  Kantorovich potentials and conclusions}\label{comp}

\begin{lmm}\label{maxcompatti} For all $n \in \N$ 
\begin{equation}\label{dualn}
\max D(c_n, \cdot)= \min K(c_n, \cdot).
\end{equation}
Moreover  $D(c_n, \cdot)$ admits a maximizer 
$(u^n_1, \dots, u^n_N)$
\end{lmm}

\begin{proof} The proof is revisited from \cite{kellerer1984}.\\
Since we have written $c_n(\bx) = \sum_{i=1}^{k_n} a_i \chi_{A^1_i}(x_1) \dots  \chi_{A^N_i}(x_N) $ we may find sets $X_1, \dots, X_N$ each of $k_n+1$ elements,  
$$ \varphi_i:\R^3 \to X_i$$
and 
$$\tilde{c}_n:X_1 \times \dots\times X_N \to \R$$
such that
$$c_n(\bx)=\tilde{c}_n (\varphi_1 (x_1), \dots, \varphi_N(x_N)).$$
If we define $\varphi:\R^{3N} \to X_1 \times \dots\times X_N$ by $\varphi=\otimes_{i=1}^N \varphi_i$ 
and $\rho_i=\varphi_{i\ \sharp} \rho \in \Pb(X_i) $, then we have that for all $\gamma \in \Pi(\rho)$ we may consider $\tilde{\gamma}:= \varphi_\sharp \gamma \in \Pi (\rho_1, \dots, \rho_N)$ and
$$ \int_{\R^{3N}} c_n d \gamma= \int_{\R^{3N}} \tilde{c}_n(\varphi(\bx)) d \gamma=  \int_{X_1 \times \dots \times X_N} \tilde{c}_n d \tilde{\gamma}.$$  
We remark that since $\tilde{c}_n$ and $\tilde{\gamma}_n$ may be identified with elements of $\R^{k_n +1} \times \cdots \times \R^{k_n+1}$ and $\rho_i$ 
with an element of $\R^{k_n+1}$, the problem may be reformulated as follows:
\begin{equation}
\left\{\begin{array}{l}
\min \tilde{c}\cdot \tilde{\gamma}, \\
P_i \tilde{\gamma}=\rho_i,\\
0 \leq \tilde{\gamma},
\end{array}
\right.
\end{equation}
where the $P_i$ form the projection matrix.
Thus the problem is a linear programming problem whose minimum value coincides with the maximum value of the dual problem
\begin{equation}
\left\{\begin{array}{l}
\max \tilde{\rho}^T\cdot \tilde{u}, \\
P^T \tilde{u}\leq \tilde{c}^T.
\end{array}
\right.
\end{equation}
It remains to identify $\tilde{u}$ with a $N$-tuple $(u_1,\dots,u_N)$ of {\it elementary} functions in ${\mathcal E}(\R^{3N})$
which maximises 
\begin{equation}\label{varmax}
 \max \{\sum_{i=1}^N \int u_i d\rho \ : \ u_i \in L^1_\rho, \ u_1(x_1)+ \dots +u_N(x_N) \leq c_n (x_1, \dots, x_N)\}. 
 \end{equation}
\end{proof}

We now prove a first property of minimiser $P$ of $K(c, \cdot)$. This requires additional assumptions on $\rho$.

Since the pointwise transportation cost diverges when two ($3-$dimensional) coordinates  get too close it is useful 
to introduce the following notations:
$$D_\alpha :=\{\bx=(x_1, \dots, x_N) \ : \ |x_i-x_j| \leq \alpha\ \mbox{for some}\ i,j \},$$
will be the closed strip around the diagonals, 
$$G:=\{\bx=(x_1, \dots, x_N) \ : \ x_i\not =x_j\ \mbox{if}\ i\not = j \}.$$
For simplicity we pose $D_0=D$. Then $G=\R^{3N}\setminus D$. Finally 
we introduce a notation for a "cube" with rounds $3-$dimensional faces 
$$Q(\bx, r):=\{\by=(y_1, \dots, y_N) \ : \  \max_i |y_i-x_i| <r\},$$
we will omit the center $\bx$ when it is the origin.

\begin{prop}\label{diagcube}
Assume that $\rho$ does not have atoms. Let $P\in \Pi(\rho)$  be a minimizer.
Then, for all $r>0$ there exists $\alpha (r)$ such that 
$$ P(D_{\alpha(r)} \cap Q(r))=0$$
\end{prop}
We first prove an elementary lemma
\begin{lmm}\label{punti} Assume that $\rho$ does not have atoms. Let $P\in \Pi(\rho)$be a plan with finite cost and let $\bx \in spt P$. 
Then there exist $\bx^2, \dots, \bx^N\in spt P$ such that 
\begin{enumerate}
\item $\bx^2, \dots, \bx^N\in G$,
\item $x^i_j \not = x^k_\sigma$  if $k \neq i$ or $\sigma \neq j$.
\end{enumerate}
\end{lmm}
\begin{proof}
First remark that if $P$ has finite cost then $P(D)=0$ and then $P$-a.a. points belong to $G$.
Then fix a vector $\overline{a}\in \R^3$ and consider the set $X^i_{\overline{a}}=\{\bx \in \R^{3N}\ : \ x_i=\overline{a}\}.$
By definition of marginals and since $\rho$ does not have atoms
$$ P(X^i_{\overline{a}})=\rho(\{\overline{a}\})=0.$$ 
Then, starting from $\bx^1=\bx$ we may choose 
$$\bx^j \in spt P \setminus (D \cup (\cup_{k < j} \cup_{\stackrel{i=1,\dots, N}{\sigma=1,\dots, N}} X^i_{x^k_\sigma})).$$ 
\end{proof}

\begin{proof}(of Proposition \ref{diagcube})
  Assume that $\bx^1=(x^1_1,\dots, x^1_N) \in D \cap spt P$ and choose points $\bx^2,\dots, \bx^N$ in $spt P$ as in Lemma 
\ref{punti}. For all choices of positive radii $r_1 \dots, r_N$ 
$$ P(Q(\bx^i, r_i))>0.$$
And later on we will choose suitable $r_i$'s. Denote by $P_i= P_{|Q(x^i, r_i)}$ and choose constants $\lambda_i\in (0, 1]$ such that
$$ \lambda_1 |P_1| = \dots = \lambda_N |P_N|$$
We then write
$$P= \lambda_1 P_1+\dots+ \lambda_N P_N  + P_R \ \ \ (P_R \ \mbox{is the remainder}).$$
We estimate from below the cost of $P$ as follows
\begin{eqnarray*}
C(P)&=& C(P_R)+ \sum_{i=1}^N \lambda_i C(P_i) \geq \\
& \ &  C(P_R)+\sum_{i=1}^N \lambda_i (\sum _{k=1}^N \sum _{k<j}\frac{1}{|x^i_k-x^i_j|+2r_i})|P_i|.
\end{eqnarray*}
Consider now the marginals $\nu^i_1, \dots, \nu^i_N$ of $\lambda_i P_i$ and build the new local plans
$$\tilde{P}_1=\nu_1^1 \times \nu_2^2\times \dots \nu^N_N, \ \tilde{P}_2=\nu_1^2 \times \nu_2^3\times \dots \nu_N^1, \dots, 
 \tilde{P}_N=\nu_1^N \times \nu_2^1\times \dots \nu_N^{N-1}.$$
To write the estimates from below it is convenient to remark that we may also write: 
$\tilde{P}_i=\nu_1^i \times \dots \nu_k^{i+k-1}\times \dots \nu_N^{i+N-1}$ where we consider the lower index $(\mbox{mod}\ N)$.
Then consider the new transport plan
$$\tilde{P}:= P_R+\tilde{P}_1+ \dots + \tilde{P}_N .$$
It is straightforward to check that the marginals of $\tilde{P}$ are the same as the marginals of $P$.
Moreover $|\tilde{P}_i|= \lambda_i |P_i|$. So we can estimate the cost of $\tilde{P}$ from above. 
\begin{eqnarray*}
C(\tilde{P})&=& C(P_R)+ \sum_{i=1}^N C(\tilde{P}_i) \leq \\
& \ &  C(P_R)+\sum_{i=1}^N  (\sum _{k=1}^N \sum _{k<j}\frac{1}{|x^k_{k+i-1}-x^j_{j+i-1}|-r_k-r_j})|\tilde{P}_i|.
\end{eqnarray*}
The final step is to choose $r_i$ for $i=1, \dots, N$ so that the $3-$dimensional faces of the cubes $Q(\bx^i, r_i)$ do not overlap and 
$$\sum_{i=1}^N  (\sum _{k=1}^N \sum _{k<j}\frac{1}{|x^k_{k+i-1}-x^j_{j+i-1}|-r_k-r_j})<\sum_{i=1}^N (\sum _{k=1}^N \sum _{k<j}\frac{1}{|x^i_k-x^i_j|+2r_i}). $$
This final condition contradicts the minimality of $P$ and it is feasible because, since $x^1_i=x^1_j$ 
for some $i$ and $j$ the right hand side is unbounded for $r_1 \to 0$ while the left hand side is bounded above for 
$r_i$ sufficiently small.

It follows that the diagonal $D$ and the  $spt P$ do not intersect and since both sets are compact inside $\overline{Q(r)}$, they must have positive distance in $Q(r)$ 
\end{proof}

Let $P_n$ denote a sequence of minimisers of $K(c_n,\cdot)$ converging to a minimiser $P$ of $K(c, \cdot)$.
We choose $0<\rc$ such that $P(Q(\rc)) =M>0$ and we consider $\alpha(\rc)$ according to Proposition \ref{diagcube} above. 
Since $Q(\rc)$ is open  we have, 
$$ \liminf_{n \to \infty} P_n (Q(\rc)) \geq M  $$ 
and then for $n$ big enough $P_n (Q(\rc)) \geq \frac{M}{2}. $
Since  $D_{\alpha(\rc)} \cap \overline{Q(\rc)}$ is closed we have 
$$ \lim_{n \to \infty } P_n(D_{\alpha(\rc)} \cap \overline{Q(\rc)})=0.$$
Then for $n$ big enough
\begin{equation}\label{bign}
 \frac{M}{4} <  P_n (Q(\rc) \setminus D_{\alpha(\rc)}). 
\end{equation}
\begin{prop} \label{controln} Assume that $\rho$ does not have atoms and let $n$ satisfies (\ref{bign}) above,  then there exists a maximiser $u$  of $D(c_n, \cdot)$ and two positive constants  $\rn$ and $\kn$ such that
$$ |u| \leq \frac{2N(N-1)^2}{\rn}- (N-1)^2 \kn.$$
\end{prop}
We will later show that we can control $\rn$ and $\kn$ uniformly in $n$.
\begin{proof}  Consider a point $(\overline{x}_1, \dots , \overline{x}_N) \in spt P_n \setminus D_{\alpha(\rc)}$ and in $Q(\rc)$.
 
We start from a maximiser  $(u_1, \dots, u_N)$ of (\ref{varmax})  and we remark that without loss of generality we may assume that 
$\displaystyle u_i(\overline{x}_i)=\frac{c (\overline{x}_1, \dots , \overline{x}_N) }{N}:=\kn$ for all $i$  and that
$$ u_i(x)= \inf \{ c(y_1, \dots, x, y_{i+1} \dots, y_N) - \sum_{k \neq i} u_k(y_k)\},$$
for all $i$ and $x$.
Then we begin by choosing $\rn < \frac{\alpha(\rc)}{2}$ and we have the following estimate. If $x \not \in \cup_{i=2} ^N B(\overline{x}_i, \rn)$
then
\begin{eqnarray}\label{firststep}
u_1(x)&\leq & c_n(x, \overline{x}_2, \dots , \overline{x}_N) - u_2(\overline{x}_2)- \dots -u_N(\overline{x}_N) \leq  \nonumber \\
& \leq & c(x, \overline{x}_2, \dots , \overline{x}_N) - u_2(\overline{x}_2)- \dots -u_N(\overline{x}_N) \leq \nonumber \\
&\leq & \frac{N(N-1)}{\rn}- (N-1) \kn.
\end{eqnarray}
We now select a, possibly smaller,  $\rn$ so that 
\begin{equation}\label{smallcontroll}
\rho( B(\overline{x}_1,\rn)) + \dots + \rho( B(\overline{x}_N,\rn)) < \varepsilon <<\frac{M}{4}. 
\end{equation}
It follows that the set
$$ \{ \by \in Q(\rc)\cap spt P_n \ | \  \by \not \in D_{\alpha(\rc)}, \ y_i \not \in B(\overline{x}_j,\rn) \ \mbox{for}\ i, j= 1, \dots, N\}$$
has positive $P_n$ measure.
Next, take $(\overline{y}_1, \dots , \overline{y}_N) $ in this last set.
Since $\overline{y}_1$ does not belong to the balls centered at $\overline{x}_i$ the estimate above holds and then 
\begin{eqnarray*}
u_2(\overline{y}_2)+ \dots+ u_N(\overline{y}_N)&=& c_n(\overline{y}_1, \dots , \overline{y}_N)- u_1(\overline{y}_1) \geq \\
&\geq & - (\frac{N(N-1)}{\rn}- (N-1) \kn).
\end{eqnarray*}
Finally, up to a division by $2$ of $\rn$  we have that for all $x \in \cup_{i=2} ^N B(\overline{x}_i, \rn)$ 
\begin{eqnarray}\label{laststep}
u_1(x)&\leq & c_n(x, \overline{y}_2, \dots , \overline{y}_N) - u_2(\overline{y}_2)- \dots - u_N(\overline{y}_N) \leq \nonumber \\
& \leq & c(x, \overline{y}_2, \dots , \overline{y}_N) - u_2(\overline{y}_2)- \dots -  u_N(\overline{y}_N) \leq \nonumber \\
&\leq & \frac{N(N-1)}{\rn}+ \frac{N(N-1)}{\rn}- (N-1)\kn.
\end{eqnarray}
This completes the estimate from above of $u_1$. The same computation holds for the other $u_i$.
The estimate from above of the $u_i$ given by  (\ref{firststep}) and (\ref{laststep}) translates in an estimates from below which holds $\rho$-a.e.. Indeed for $\rho-$a.e. $x$ there holds
$$u_1(x) =\inf \{c_n(x , x_2,\dots, x_n)- u_2(x_2)-\dots-u_N(x_N) \} \geq \ (N-1)^2 \kn -\frac{2N(N-1)^2}{\rn}.$$
It remains to remark that 
$$u(x)=\frac{1}{N} \sum_{i=1}^N u_i(x)$$
is a Kantorovich potential for $c_n$ and it satisfies the required estimate.
\end{proof}
\begin{prop}\label{controlunif}
The constants $\rn$ and $\kn$ in Proposition \ref{controln} can be controlled  uniformly in $n$.
\end{prop}
\begin{proof} 
Consider a point $\bx=(\overline{x}_1, \dots , \overline{x}_N) \in spt P \setminus D_{\alpha(\rc)}$ and in $Q(\rc)$.
Then consider $r$ such that 
\begin{enumerate}
\item $r < \frac{\alpha(\rc)}{2}$,
\item $N \sum_i \rho (B(\overline{x}_i, r) < \frac{M}{4}$.
\end{enumerate}
Now we will show that in the construction of Proposition  \ref{controln}  we may choose  $\displaystyle k_n \to \frac{ c(\overline{x}_1, \dots , \overline{x}_N)}{N}$ and
$\displaystyle \rn=\frac{r}{4}$. 
Infact, since $P_n \stackrel{*}{\rightharpoonup}P$ we may choose in the estimates of Proposition \ref{controln}  a sequence of points $\bx^n=(\overline{x}_1^n, \dots , \overline{x}_N^n) \to(\overline{x}_1, \dots , \overline{x}_N). $
The convergence of $\bx^n $ to $\bx$ already gives the required convergence of $\kn$. Moreover, we also have that for $n$ big enough $B(\overline{x}_i^n, \frac{r}{2}) \subset B(\overline{x}_i, r)$ for all $i$ and then  (\ref{smallcontroll}) is satisfied. Finally in Proposition \ref{controln} we divided $\rn$ by 2 to have 
some distance between $\overline{x}_i$ and $\overline{y}_j$ and this last division brings us to $\displaystyle \rn=\frac{r}{4}$. 
\end{proof}

\subsection{Conclusion and remarks}
\begin{thm} The following duality holds
\begin{multline}\label{dual}
\min_{\Pi(\rho)} \int_{\R^{Nd}} \sum_{1\leq i <j \leq N} \frac{1}{|x_i-x_j|} d P(x_1, \dots, x_N)\\
= \sup\left\{N\int u(x) d \rho (x) \ : \  u \in L^1_\rho\ \mbox{and}\ u(x_1)+\dots+u(x_n) \leq c(x_1,\dots,x_n)\right\},
\end{multline}
and the right-hand side of equation (\ref{dual}) above admits a bounded maximizer.
\end{thm}
\begin{proof}
By  Prop. \ref{gammamisure} 
$$\min_{\Pi(\rho)} \int_{\R^{Nd}} \sum_{1\leq i <j \leq N} \frac{1}{|x_i-x_j|} d P(x_1, \dots, x_N)= \lim_{n \to \infty} \min K(c_n,P).$$ 
By Propositions \ref{gammamassimi} , \ref{maxcompatti}, \ref{controln} and \ref{controlunif}
$$  \min K(c_n,P)= \max D(c_n,u),$$
$$\sup\left\{N\int u(x) d \rho (x) \ \left|\begin{array}{c} u \in L^1_\rho\ \mbox{and}\\\ u(x_1)+\dots+u(x_n) \leq c(x_1,\dots,x_n)\end{array}\right.\right\}=\lim_{n \to \infty} \max D(c_n,u), $$
and there exists a sequence $\{u_n\}$  where $u_n$ is a maximizer of $D(c_n,u)$  and  where $|u_n|$ is uniformly bounded, thus weakly compact in $L^1_\rho$.
\end{proof}
\begin{corol} Assume that $\rho$ does not have atoms then  there exists $0<\alpha$ such that for every minimizer $P$ of $K(c, \cdot)$ 
$$P(D_\alpha)=0$$
\end{corol}
\begin{proof} For every Kantorovich potential $u$ exploiting the duality (or complementary slackness) relations we obtains that for every 
minimizer $P$ of $K(c, \cdot)$ 
$$ u(x_1)+\dots+u(x_N)=c(x_1, \dots, x_N) \ \ P-a.e.$$
Since there exists a bounded Kantorovich potential we obtain the conclusion.
\end{proof}

\begin{rmk}\label{nokellerer}
Driven by the interest in some application we have chosen to present the result in the case in the case of Coulomb cost. However fixing some constants and exponent (in particular in  Lemma 3.4, Prop. 3.7 and Subsection \ref{comp}) and the definition of {\it bad set}  the same result may be proved for several costs which are lower semi-continuous and bounded from below. Among them 
$$ \sum_{1\leq i < j \leq N} \frac{1}{|x_i-x_j|^s} .$$
The case of even less regular costs has been considered in \cite{beiglbock2012general} but in that case it is necessary to give a different interpretation of the problems.

In \cite{kellerer1984} Kellerer considered a duality theory for multimarginal problems for a very wide class of costs. 
However in order to have existence of maximizers for the dual problems (see Theorem 2.21 in \cite{kellerer1984}) it is required that the cost 
$c$ be controlled by a direct sum of functions in $L^1_\rho$ also from above. In the case of the Coulomb cost this would read
$$ \sum_{1\leq i <j \leq N} \frac{1}{|x_i-x_j|} \leq u(x_1)+\dots+u(x_N) $$
and this is not possible since the right-hand side allows for $x_1=\dots=x_n$.  
\end{rmk}

\section*{Acknowledgement}
The research of the author is part of the project 2010A2TFX2 {\it Calcolo delle Variazioni}  financed by the Italian Ministry of Research
 and is partially financed  by the {\it ``Fondi di ricerca di ateneo''}  of the  University of Pisa. 
 
 The research of the author is part of the project  {\it "Analisi puntuale ed asintotica di energie di tipo non locale collegate a modelli della fisica"}  of the  GNAMPA-INDAM.

The author wish to thank Giuseppe Buttazzo and Aldo Pratelli for several discussions on the research presented here.

The authors also would like to thank the anonymous reviewer for his valuable comments and suggestions to improve the paper.

%%%%%%%%%%%%%%%%%%%%%%%%%%%%%%%%%%%%%%%%%%%%%%%%%%%%%%%%%%%%%%%%%%%%%%%%%%
\bibliographystyle{plain}

\end{document}